\documentclass{amsproc}
\usepackage{eurosym}
\usepackage{amssymb}
\usepackage{amsfonts}
\usepackage{cmtt}
\usepackage{color}
\usepackage{graphicx}					

\usepackage{calrsfs}
\DeclareMathAlphabet{\pazocal}{OMS}{zplm}{m}{n}

\DeclareMathOperator{\modu}{mod}

\DeclareMathOperator{\vol}{vol}			
\theoremstyle{plain}
\newtheorem{theorem}{Theorem}

\newtheorem{corollary}[theorem]{Corollary}
\newtheorem{definition}[theorem]{Definition}

\newtheorem{lemma}[theorem]{Lemma}

\newtheorem{fact}[theorem]{Fact}
\newtheorem{remark}[theorem]{Remark}

\renewcommand{\Pr}{\mathbb{P}}
\newcommand{\E}{\mathbb{E}}

\theoremstyle{remark}

\newcommand{\Set}{S}
\newcommand{\Setbar}{\bar{S}}
\newcommand{\Part}{{\mathcal A}}
\newcommand{\set}{{\bf s}}
\newcommand{\dd}{{\bf d}}
\newcommand{\qq}{{\bf q}}

\newcommand{\eSet}{e(\Set)}
\newcommand{\eSetbar}{e(\Setbar)}
\newcommand{\eSetSetbar}{e(\Set,\Setbar)}
\newcommand{\whp}{with high probability}
\newcommand{\Bin}[2]{\textrm{Bin}\left(#1,#2\right)}
\newcommand{\Pra}[1]{\Pr\left(#1\right)}

\newcommand{\Gnp}{G(n,p)}
\newcommand{\modA}{\modu_{\Part}}
\newcommand{\BB}{{\mathcal B}}

\newcommand{\err}{{\bf Er}}

\newcommand{\superscript}[1]{\ensuremath{^{\textrm{#1}}}}

\def\wu{\superscript{*}}
\def\wg{\superscript{$\star$}}

\begin{document}
	
	\title[New bounds on the modularity of $G(n,p)$]{New bounds on the modularity of $\mathbf{G(n,p)}$} 
	
	\author[K.~Rybarczyk]{Katarzyna Rybarczyk\wu\footnote{\wu kryba@amu.edu.pl, Adam Mickiewicz University, Pozna\' n, Poland}} 
	
	\author[M.~Sulkowska]{Ma{\l}gorzata Sulkowska\wg\footnote{\wg malgorzata.sulkowska@pwr.edu.pl,
	Wroc{\l}aw University of Science and Technology, Department of Fundamentals of Computer Science, Poland}}

\thanks{This research was funded in part by National Science Centre, Poland, grant OPUS-25 no 2023/49/B/ST6/02517.}


\begin{abstract}
 Modularity is a parameter indicating the presence of community structure in the graph. Nowadays it lies at the core of widely used clustering algorithms. We study the modularity of the most classical random graph, binomial $\Gnp$. In 2020 McDiarmid and Skerman proved, taking advantage of the spectral graph theory and a specific subgraph construction by Coja-Oghlan from 2007, that there exists a constant $b$ such that with high probability the modularity of $\Gnp$ is at most $b/\sqrt{np}$. The obtained constant $b$ is very big and not easily computable. We improve upon this result showing that a constant under $3$ may be derived here. Interesting is the fact that it might be obtained by basic probabilistic tools. We also address the lower bound on the modularity of $\Gnp$ and improve the results of McDiarmid and Skerman from 2020 using estimates of bisections of random graphs derived by Dembo, Montanari, and Sen in 2017. 
\end{abstract}

\maketitle

\section{Introduction}

Identifying communities in a graph is a task of a great importance. It has a number of practical applications, varying from identifying the groups of common interests in social networks, through classifying fake news or spam, searching articles on related topic, identifying proteins with the same biological functions, optimizing large technological infrastructures, up to networks visualization \cite{KaminskiBook,NewmanBook}. Therefore various algorithmic approaches to graph clustering arose over the years. Just to mention a few: some are based on random walks, others on hierarchical clustering, label propagation, or spectral bisection method  \cite{KaminskiBook}. Recently widely used algorithms that offer a good trade-off between the time complexity and the quality of identified communities are Louvain algorithm \cite{BlGuLaLe08} and its follow-ups -  Leiden algorithm \cite{TrWaEc19}, and Tel-Aviv algorithm \cite{GiSh_23}. They are similar in nature, i.e., they heuristically try to optimize a parameter called \emph{modularity}. Modularity (introduced by Newman and Girvan in 2004 \cite{Newman2004}) serves as a measure of the presence of community structure in the graph. As mentioned, it is commonly utilized in practice, however many theoretical questions about it, even for very classical families of graphs, still remain unanswered. 

We study the modularity function of the most classical random graph, binomial $\Gnp$. First such results were given by McDiarmid and Skerman just in 2020 \cite{McDiarmid2020}. One can check the appendix therein for a summary of results on modularity for various graph classes (\cite{Bolla2015,Brandes2008,Majstorovic2014,McDiarmidScerman2018,McDiarmidSkerman_dense2023,Prokhorenkova2017}) and further references. More recent discoveries one finds in \cite{ChFoSk21} by Chellig, Fountoulakis and Skerman (for random graphs on the hyperbolic plane), \cite{LaSu23} by Laso{\'n} and Sulkowska (for minor-free graphs), \cite{LiMi22} by Lichev and Mitsche (for $3$-regular graphs and graphs with a given degree sequence), \cite{Rybarczyk2025} by Rybarczyk and Sulkowska and \cite{McDRSS_2024} by McDiarmid, Rybarczyk, Skerman, and Sulkowska (for preferential attachment graphs), and \cite{Rybarczyk2025randomintersectiongraphs} by Rybarczyk for random intersection graphs.

The authors of \cite{McDiarmid2020} showed that the modularity of $\Gnp$ with high probability decreases at the rate $1/\sqrt{np}$ by $n$ tending to infinity (note that $np$ may be asymptotically interpreted as an average degree). In this paper we study the upper and the lower bound on the modularity of $\Gnp$. We show how, using basic probabilistic methods, significantly improve the upper bound presented in \cite{McDiarmid2020}. We also improve upon the lower bound from \cite{McDiarmid2020} using the estimates of bisections of random graphs from \cite{Dembo2017}.

Notation, previous results, and our findings are presented in Section~\ref{Sec:notation_result}. In Section~\ref{Sec:main_upper_bound} we derive a new upper bound on the modularity of $\Gnp$. Doing that, we use a technical lemma bounding the number of edges of induced subgraphs of $\Gnp$ - its proof may be found in Section \ref{Sec:proof_of_lemma}. The last section is devoted to the lower bound on $\Gnp$.

\section{Notation and statement of results} \label{Sec:notation_result}

Let $\mathbb{N}$ denote the set of natural numbers, $\mathbb{N} = \{1,2,3,\ldots\}$. For $n \in \mathbb{N}$ let $[n]=\{1,2,\ldots,n\}$. If a random variable $X$ follows a binomial distribution with parameters $n \in \mathbb{N}$ and $p \in [0,1]$ we write $X \sim \Bin{n}{p}$. For the functions $f, g: \mathbb{N} \rightarrow \mathbb{R}$ we use the notation $f(n) = o(g(n))$ and $f(n) \ll g(n)$ interchangeably and understand it as $\lim_{n \rightarrow \infty} f(n)/g(n) = 0$. Unless otherwise stated, $o(\cdot)$ relates to $n \to \infty$. By $f(n) \sim g(n)$ we mean $\lim_{n \rightarrow \infty} f(n)/g(n) = 1$.
We say that an event $\mathcal{E} = \mathcal{E}_n$ occurs with high probability (whp) if the probabilities $\Pra{\mathcal{E}_n}$ tend to $1$ as $n$ tends to infinity. All inequalities in the paper are true for $n$ large enough and, for clarity of notation, we omit floors and ceilings when it does not influence a reasoning.

The graphs considered are finite, simple and undirected. Thus a graph is a pair $G=(V(G),E(G))$, where $V(G)$ is a finite set of vertices and $E(G) \subseteq V(G)^{(2)}$ where, for any set $\Set$, $S^{(2)}$ stands for a set of all $2$-element subsets of $\Set$. 
Let $e(G) = |E(G)|$ and for $\Set \subseteq V(G)$ set $\Setbar = V(G) \setminus \Set$, $\eSet = |\{e \in E(G) \cap \Set^{(2)}\}|$, and $\eSetSetbar = |\{e \in E(G): e \cap \Set \neq \emptyset \wedge e \cap \Setbar \neq \emptyset\}|$. The degree of a vertex $v \in V(G)$ in $G$, denoted by $\deg(v)$, is the number of edges to which $v$ belongs, 
i.e., $\deg(v) = |\{e \in E(G): v \in e\}|$. 
We define the volume of $ \Set \subseteq V$ in $G$ by $\vol(\Set) = \sum_{v \in \Set} \deg(v)$. By the volume of the graph, $\vol(G)$, we understand $\vol(V(G))$.

We also introduce several basic definitions originating from spectral graph theory. We do not make use of its tools in our proofs, however we refer to the results that do that. For a graph $G$ on the vertex set $[n]$ by $A=\{a_{i,j}\}_{i,j \in [n]}$ denote its adjacency matrix (i.e. $a_{i,j} = 1$ if $\{i,j\} \in E(G)$ and $a_{i,j} = 0$ otherwise) and by $D=\{d_{i,j}\}_{i,j \in [n]}$ its diagonal degree matrix (i.e. $d_{i,i} = \deg(i)$ and  $d_{i,j} = 0$ for $i \neq j$). A matrix ${\mathcal{L}}=I-D^{-1/2}AD^{-1/2}$ is called a normalized Laplacian of $G$ (here $D^{-1/2}$ stands for a diagonal matrix with diagonal entries $\deg(i)^{-1/2}$ if $\deg(i) \neq 0$ and $0$ otherwise). By $0 = \lambda_0 \le \lambda_1 \le \ldots \le \lambda_{n-1} \le 2$ denote the eigenvalues of $\mathcal{L}$ (see \cite{ChungBook}). Then $\bar\lambda(G) = \max_{i \neq 0}|1-\lambda_i| = \max\{|1-\lambda_1|, |1-\lambda_{n-1}|\}$ is called the spectral gap of $G$. 


We work with the classical binomial random graph and state its definition as it appears in \cite{JLRBook} by Janson, {\L}uczak, and Ruci{\'n}ski.

\begin{definition}[Binomial random graph $G(n,p)$] 
	Let $n \in \mathbb{N}$ and $p \in [0,1]$. The binomial random graph is defined by taking the set of all graphs on vertex set $[n]$ as a sample space $\Omega$, defining $\mathcal{F} = 2^{\Omega}$ as an event space, and for $G \in \Omega$ setting
	\[
	\Pra{G} = p^{e(G)}(1-p)^{{n \choose 2} - e(G)}.
	\]
\end{definition}
Obviously, $G(n,p)$ may be seen as a result of ${n \choose 2}$ independent Bernoulli trials, one for each pair of vertices from $[n]$, with the probability of success (which is including a particular edge in the graph) equal to $p$. Throughout the paper we consider $p$ as a function of $n$, $p=p(n)$, and often denote $\dd = \dd(n) = np$. Moreover, with $\Set \subseteq V(G(n,p))$ we often associate a function $\set = \set(n)$ such that $|\Set| = \set n$.

We focus on bounding a graph parameter called modularity for $\Gnp$. Its formal definition is given just below.

\begin{definition}[Modularity, \cite{Newman2004}] \label{def:modularity}
	Let $G=(V(G),E(G))$ be a graph with at least one edge. For a partition $\Part$ of $V(G)$ define a modularity score of $G$ as
	\begin{equation} \label{Eq:DefModularity1}
		\modA(G)=\sum_{\Set\in\Part}\left(\frac{\eSet}{e(G)}-\frac{\vol(\Set)^2}{\vol(G)^2}\right) =\sum_{\Set\in\Part} \frac{4\eSet e(G) - \vol(\Set)^2}{4e(G)^2}.
	\end{equation}
	Modularity of $G$ is given by
	\[
	\modu(G) = \max_{\Part}\modA(G),
	\]
	where maximum runs over all the partitions of the set $V(G)$. For a graph $G$ with no edges set $\modu(G)=0$.
\end{definition}

A single summand of the modularity score is the difference between the fraction of edges within $\Set$ and the expected fraction of edges within $\Set$ if we considered a certain random multigraph on $V(G)$ with the expected degree sequence given by $G$ (see, e.g.,~\cite{Kaminski2019}). It is easy to check that $\modu(G) \in [0,1)$. Higher values of $\modu(G)$ are meant to indicate the clearer community structure in $G$. 

As already mentioned, the first results on the modularity of $\Gnp$ appeared in $2020$ and are due to McDiarmid and Skerman \cite{McDiarmid2020}. Among others, they established a general relation between the modularity of $G$ and a spectral gap of $G$.

\begin{lemma}[Lemma 6.1 of \cite{McDiarmid2020}] \label{Lemma:mod_spectral_gap}
	Let $G$ be a graph with at least one edge and no isolated vertices. Let $\bar\lambda(G)$ be a spectral gap of $G$. Then
	$$
		\modu(G) \leq \bar\lambda(G).
	$$
\end{lemma}

The spectral gap of $\Gnp$ was already investigated many years earlier. In \cite{Chung2003} from $2003$ Chung, Lu, and Vu studied eigenvalues of random graphs with given expected degrees. Their results applied to $\Gnp$ may be formulated as follows (compare also (1.2) in \cite{Coja-Oghlan2007} by Coja-Oghlan).
\begin{theorem}[Theorem 3.6 of \cite{Chung2003}] \label{Thm:ChungLuVu}
	Let $n \in \mathbb{N}$, $p = p(n) \in (0,1]$. Set $\dd = \dd(n) =np$ and assume that $\dd \gg (\ln{n})^2$. Let also $\Gnp$ be a binomial random graph and $\bar{\lambda}(\Gnp)$ its spectral gap. Then {\whp}
	$$
		\bar{\lambda}(\Gnp) \leq (1+o(1))\frac{4}{\sqrt{\dd}}.
	$$
\end{theorem}

Obviously, by Lemma \ref{Lemma:mod_spectral_gap} and Theorem \ref{Thm:ChungLuVu} one gets the whp upper bound on $\modu(\Gnp)$ of the order $1/\sqrt{\dd}$ for a certain range of $\dd$.

\begin{corollary} \label{Cor:spectral_4}
	Let $n \in \mathbb{N}$, $p = p(n) \in (0,1]$ and let $\Gnp$ be a binomial random graph. Set $\dd = \dd(n) =np$ and assume that $\dd \gg (\ln{n})^2$. Then {\whp}
	$$
		\modu(\Gnp) \leq (1+o(1))\frac{4}{\sqrt{\dd}}.
	$$
\end{corollary}

The $1/\sqrt{\dd}$ growth rate of $\modu(\Gnp)$ when $\dd$ is above $1$ and not too big was first conjectured in 2006 by Reichardt and Bornholdt \cite{Reichardt2006}. In 2020 McDiarmid and Skerman formally confirmed it proving the following result.

\begin{theorem}[Theorem 1.3 and Theorem 4.1 of \cite{McDiarmid2020}] \label{Thm:McDiar_Sker}
	Let $n \in \mathbb{N}$, $p = p(n) \in (0,1]$, and $\Gnp$ be a binomial random graph. Set $\dd = \dd(n) =np$. Then there exists a constant $b$ such that with high probability
	$$
		\modu(\Gnp) \le \frac{b}{\sqrt{\dd}}.
	$$
	Moreover, there exist constants $c_0$ and $a$ such that 
	\begin{enumerate}
		\item{if $\dd$ satisfies $c_0 \leq \dd < n - c_0$ for sufficiently large $n$, then {\whp}
		\[
			\modu(\Gnp) \ge \frac{1}{5}\frac{\sqrt{1-p}}{\sqrt{\dd}},
		\]
		}
		\item{if $\dd$ satisfies $1 \leq \dd < n - c_0$ for sufficiently large $n$, then {\whp}
		\[
			\modu(\Gnp) \ge a \frac{\sqrt{1-p}}{\sqrt{\dd}}.
		\]
		}
	\end{enumerate}
\end{theorem}

First, we concentrate on the upper bound on $\modu(\Gnp)$. The authors of \cite{McDiarmid2020} derived their upper bound in Theorem \ref{Thm:McDiar_Sker} using quite sophisticated tools. Following the construction from~\cite{Coja-Oghlan2007} they extracted a particular induced subgraph $H$ of $\Gnp$, not possessing vertices of very high or very low degrees. Next, they took advantage of the result for a spectral gap of $H$ from \cite{Coja-Oghlan2007} and applied Lemma \ref{Lemma:mod_spectral_gap} to get the whp upper bound on the modularity of $H$. Finally, they proved that modularity is robust under removing not too many edges from the graph to justify that the modularity of $\Gnp$ {\whp} will not differ too much from the modularity of $H$. From their proof it follows that the constant $b$ in Theorem \ref{Thm:McDiar_Sker} is huge and, in practice, not easily computable.

In this paper we show that the better upper bound on $\modu(\Gnp)$, with a specific small constant under $3$, may be obtained using basic probabilistic tools. Our proof gives also the straightforward insight into how the $1/\sqrt{\dd}$ growth rate of the uppper bound is related to the fluctuations of the number of edges of the induced subgraphs of $\Gnp$. The first one of two main theorems of this paper is stated below.

\begin{theorem}\label{Thm:Main}
	Let $n \in \mathbb{N}$, $p = p(n) \in (0,1)$, and $G(n,p)$ be a binomial random graph. 
	Set $\dd = \dd(n) =np$ and assume that $16.17 \le \dd \ll n$. 
	Then with high probability 
	$$
		\modu(\Gnp)\le \left(\frac{3+2\sqrt{2}}{2}\right) \frac{1}{\sqrt{\dd}} \approx \frac{2.91}{\sqrt{\dd}}.
	$$
\end{theorem}

\begin{remark}
	The bound $\dd \ge 16.17$ in the above theorem is the effect of a trade-off between the final value in the upper bound on $\modu(\Gnp)$ and the smallest value of $\dd$ for which the proof is valid. We could as well set the parameters in the proof (i.e., in Lemma~\ref{Lem:Fluktuacje} set the lower bound on $C$ slightly smaller than $2.1$) to get that for any $\dd \ll n$ with high probability 
	$$
	\modu(\Gnp)\le \left(\frac{(3+2\sqrt{2})\, 2.1}{4}\right) \frac{1}{\sqrt{\dd}} \approx \frac{3.06}{\sqrt{\dd}}.
	$$
	On the other hand, if we consider $\dd \to \infty$ as $n \to \infty$, we can get that for any $C>2\sqrt{\ln{2}}$ {\whp}
	$$
		\modu(\Gnp)\le \left(\frac{(3+2\sqrt{2})C}{4}\right) \frac{1}{\sqrt{\dd}}. 
	$$
	Note that $\frac{(3+2\sqrt{2})2\sqrt{\ln{2}}}{4} \approx 2.43$. Further comments on this fact one finds in Section~\ref{Sec:proof_of_lemma} in Remark~\ref{Rem:d_to_infty}.
\end{remark}

%
%
%
%
%

Next, we turn to the lower bound on $\modu(\Gnp)$. Lower bounds on modularity are frequently obtained by constructing a partition $\Part = \{\Set, \Setbar\}$ that minimizes $\eSetSetbar$, and, in many cases, consequently increases the values of $\eSet$ and $\eSetbar$. This technique has been successfully applied in a number of papers; see, for example, \cite{Agdur2023, McDiarmidScerman2018, McDiarmid2020, Prokhorenkova2017}. Using this approach, and leveraging known results about bisections of random graphs by Dembo, Montanari, and Sen~\cite{Dembo2017}, we obtain the following result.

\begin{theorem}\label{Thm:Main2}
	Let $n \in \mathbb{N}$, $\dd\ge 1$, $p = p(n) = \dd/n$, and $G(n,p)$ be a binomial random graph. 
	Then with high probability, for large $\dd$
	\[
	\modu(\Gnp) \ge  \frac{{P_*}+o_{\dd}(1)}{\sqrt{\dd}}, 
	\] 
	where $P_*=0.76321 \pm 0.00003 $ and $o_\dd(\cdot)$ relates to $\dd\to\infty$.  
\end{theorem}


For completeness let us also comment on the value of $\modu(\Gnp)$ when $\dd \le 1 + o(1)$. 
Then, if only $\dd n \to \infty$, the modularity score of the partition dictated by connected components {\whp} tends to $1$. The proof of this fact may be found in \cite{McDiarmid2020} and is stated formally below.

\begin{theorem} [Theorem 1.1 of \cite{McDiarmid2020}]
	Let $n \in \mathbb{N}$, $p = p(n) \in (0,1)$, and $G(n,p)$ be a binomial random graph. Set $\dd = \dd(n) =np$ and assume that $\dd n \to \infty$ and $\dd \le 1+o(1)$. Then {\whp}
	$$
		\modu(\Gnp) = 1-o(1).
	$$
\end{theorem}

\section{Upper bound on the modularity of $\mathbf{G(n,p)}$ -- proof of Theorem \ref{Thm:Main}} \label{Sec:main_upper_bound}

First, it will be useful to observe that, given a graph $G$ and a partition $\Part$ of $V(G)$, the formula for $\modA(G)$ may be rewritten in terms of $\eSet$, $\eSetbar$ and $\eSetSetbar$, where $\Set \in \Part$.

\begin{fact} \label{Fact:mod_as_edges}
	Let $G$ be a graph with at least one edge. Let $\Part$ be a partition of~$V(G)$. Then
	\[
		\modA(G)=\sum_{\Set\in\Part}
		\frac{4 \eSet\eSetbar - \eSetSetbar^2}
		{4e(G)^2}.
	\]
\end{fact}
\begin{proof}
	We have
	\begin{align*}
		\vol(\Set)&=2\eSet + \eSetSetbar \quad \quad \quad \quad \text{ and }\\
		e(G)&=\eSet + \eSetSetbar + \eSetbar,
	\end{align*}
	therefore
	\begin{align*}
		4e(G)\eSet &= 4\eSet^2 + 4\eSet\eSetSetbar + 4 \eSet\eSetbar \quad \text{ and }\\
		\vol(\Set)^2 &= 4\eSet^2 + 4\eSet\eSetSetbar +  \eSetSetbar^2.
	\end{align*}
	Consequently
	\begin{equation*}
		4e(G)\eSet-\vol(\Set)^2= 4 \eSet\eSetbar - \eSetSetbar^2
	\end{equation*}
	which substituted to formula \eqref{Eq:DefModularity1} from Definition \ref{def:modularity} gives
	\begin{equation}\label{Eq:DefModularity2}
		\modA(G)=
		\sum_{\Set\in\Part}
		\frac{4 \eSet\eSetbar - \eSetSetbar^2}
		{4e(G)^2}.
	\end{equation}	
\end{proof}

Next, note that for $\Set \subseteq V(\Gnp)$ with $|\Set|=\set n$ we have $\eSet\sim\Bin{\binom{\set n}{2}}{p}$, $\eSetbar\sim \Bin{\binom{(1-\set) n}{2}}{p}$, and $\eSetSetbar\sim \Bin{\set(1-\set)n^2}{p}$, thus by $\dd=np$ we get $\E[\eSet] = \frac{\set^2\dd n}{2} - \frac{\set \dd}{2}$, $\E[\eSetbar] = \frac{(1-\set)^2\dd n}{2} - \frac{(1-\set) \dd}{2}$, and $\E[\eSetSetbar] = \set(1-\set)n\dd$. Therefore the following bounds on $\eSet$, $\eSetbar$, and $\eSetSetbar$ may be derived mainly by Chernoff's inequality.


\begin{lemma}\label{Lem:Fluktuacje}
	Let $n \in \mathbb{N}$, $p = p(n) \in [0,1]$ and $G(n,p)$ be a binomial random graph. Let also $C \geq 1.999$. Set $\dd = \dd(n) =np$, and let $\dd \geq 16.17$. 
	Then we get that {\whp} for all $\Set\subseteq V(\Gnp)$, with $\set = \set(n) =|\Set|/ n$, 
	\begin{align}
		\label{Eq:Fluktuacje1}
		\eSet &\le \set\left(\set + C \dd^{-1/2}\right)  \frac{n\dd}{2} \quad \quad \quad \quad \quad \quad \quad \quad \text{ and } \\
		\label{Eq:Fluktuacje2}
		\eSetbar &\le (1-\set)\left((1-\set) + C \dd^{-1/2}\right)  \frac{n\dd}{2} \quad \quad \quad \text{ and } \\
		\label{Eq:Fluktuacje3}
		\eSetSetbar &\ge \left(\set(1-\set) -  C \sqrt{\set(1-\set)} \dd^{-1/2}\right) n\dd.
	\end{align}
\end{lemma}

Now we show how Lemma \ref{Lem:Fluktuacje} implies Theorem \ref{Thm:Main}. The proof of Lemma \ref{Lem:Fluktuacje} one finds in Section \ref{Sec:proof_of_lemma}.

\begin{proof}[Proof of Theorem \ref{Thm:Main}]
	Set $C = 1.999$. By Lemma \ref{Lem:Fluktuacje} we get that {\whp} for all $\Set\subseteq V(\Gnp)$, where $\set = \set(n) =|\Set|/ n$,
	\begin{align} \label{Eq:inner_concentr}
		\frac{4 \eSet\eSetbar - \eSetSetbar^2}
		{(n\dd)^2}
		&\le 
		\set(1-\set)\left(\set + C  \dd^{-1/2}\right)\left((1-\set) + C  \dd^{-1/2}\right) \nonumber \\
		&\quad-\left(\set(1-\set) -  C \sqrt{\set(1-\set)} \dd^{-1/2}\right)^2 \nonumber \\
		&=\set^2(1-\set)^2+C\set(1-\set)\dd^{-1/2}+C^2\set(1-\set)\dd^{-1} \nonumber \\
		&\quad-\set^2(1-\set)^2+2C(\set(1-\set))^{3/2}\dd^{-1/2}-C^2\set(1-\set)\dd^{-1}\\
		&=C \set (1-\set) \left(1+2(\set(1-\set))^{1/2}\right)\dd^{-1/2} \nonumber\\
		&\le \set\frac{(3+2\sqrt{2})}{4} C  \dd^{-1/2}. \nonumber
	\end{align}
	Here the last line follows by finding a supremum of the function $f(\set)=(1-\set)\left(1+2(\set(1-\set))^{1/2}\right)$ on the interval $(0,1)$.
	
	Note that $e(\Gnp) \sim \Bin{{n \choose 2}}{p}$, $\E[e(\Gnp)] = \frac{n \dd}{2}-\frac{\dd}{2}$ thus, by the concentration of a binomial random variable, we may write that {\whp} $e(\Gnp) = \left(1+o(1)\right)\frac{n\dd}{2}$. Let $\Part$ be any vertex partition of $V(\Gnp)$. Recall that $C=1.999$. By \eqref{Eq:inner_concentr} 
	we get that {\whp}
	\begin{align*}
		\modA (\Gnp)&= 
		\sum_{\Set\in\Part}
		\frac
		{4 \eSet\eSetbar - \eSetSetbar^2}
		{4e(\Gnp)^2}\\
		&\le (1+o(1)) \frac{(3+2\sqrt{2})1.999}{4}  \dd^{-1/2} \sum_{\Set\in\Part}\frac{|\Set|}{n}\\
		&=(1+o(1)) \frac{(3+2\sqrt{2})1.999}{4} \dd^{-1/2}.	
	\end{align*}
	Since the above formula is valid for any vertex partition $\Part$ of $V(\Gnp)$, it is in particular true for a vertex partition $\Part_{opt}$ realizing the modularity of $\Gnp$. Thus {\whp} 
	\[
		\modu(\Gnp) = \modu_{\Part_{opt}}(\Gnp) \le  \frac{(3+2\sqrt{2})}{2} \frac{1}{\sqrt{\dd}}. 
	\]
\end{proof}

\section{Bounds on $\mathbf{\eSet}$, $\mathbf{\eSetbar}$, and $\mathbf{\eSetSetbar}$ -- proof of Lemma \ref{Lem:Fluktuacje}} \label{Sec:proof_of_lemma}

In this section we derive the bounds on $\eSet$, $\eSetbar$, and $\eSetSetbar$ presented in Lemma~\ref{Lem:Fluktuacje}. The proof is based on quite basic tools, 
mainly on Chernoff's inequality. We state it as it appears in Theorem 2.1 of \cite{JLRBook} by Janson, {\L}uczak, and Ruci{\'n}ski.


\begin{lemma}[Chernoff's inequality] \label{Lemma:Chernoff}
Assume that $X$ follows a binomial distribution with the expected value $\mu$. Let also $\varphi(y) = (1+y) \ln{(1+y)-y}$ for $y \ge 0$. Then for any  $t\ge 0$
\begin{align}
	\Pra{X\ge \mu + t}&\le \exp\left(-\mu \, \varphi\left(\frac{t}{\mu}\right)\right)\le \exp\left(-\frac{t^2}{2(\mu+t/3)}\right) \quad \text{ and } \label{Eq:Chernoff_upper_tail}\\ 
	\Pra{X\le \mu - t}&\le \exp\left(-\frac{t^2}{2\mu}\right). \label{Eq:Chernoff_lower_tail}
\end{align}
\end{lemma}

The following technical lemma (its proof may be found in the Appendix) will be applied in the proof of Lemma \ref{Lem:Fluktuacje}.

\begin{lemma} \label{Lemma_f(x,d)}
	For $x,y,z>0$ define 
	$$
		f(x,y,z) = \frac{xy}{2} \varphi\left( \frac{z}{x} \right) - \left(\ln{\frac{y}{x}}+1\right)
		\quad
		\text{ and } \quad
		g(x,z)=\frac{x^2}{2}\varphi\left( \frac{z}{x} \right),
	$$
	where $\varphi:[0,\infty] \rightarrow \mathbb{R}$ is as in Lemma \ref{Lemma:Chernoff}. Then for all $0<x\le y/3$ 
	\[
		f(x,y,z) > 0.001 \quad \text{ for } \quad z \ge 1.999, y \ge 3.95
	\]
	and 
	\[
		g(x,z)>\ln 2 + 0.01 \quad \text{ for } \quad z \ge 1.999, x \ge 1.34.
	\]
\end{lemma}

\begin{proof} [Proof of Lemma \ref{Lem:Fluktuacje}]
	For $\Set \subseteq V(\Gnp)$ set $\set = \set(n) = |\Set|/ n$ and define the following events
	\begin{align*}
		\BB_{\Set} &= \left\{ \exists \Set \subseteq V(\Gnp)  \quad \eSet > \set\left(\set + C \dd^{-1/2}\right)  \frac{n\dd}{2}\right\}, \\
		\BB_{\Setbar} &= \left\{ \exists \Set \subseteq V(\Gnp)  \quad \eSetbar > (1-\set)\left((1-\set) + C \dd^{-1/2}\right)  \frac{n\dd}{2} \right\}, \\
		\BB_{\Set\Setbar} &= \left\{ \exists \Set \subseteq V(\Gnp)  \quad \eSetSetbar < \left(\set(1-\set) -  C \sqrt{\set(1-\set)} \dd^{-1/2}\right) n\dd \right\}.
	\end{align*}
	Note that our aim is to show that $\Pra{\BB_{\Set}} = o(1)$, $\Pra{\BB_{\Setbar}} = o(1)$, and $\Pra{\BB_{\Set\Setbar}} = o(1)$.
	
	We start with analyzing $\Pra{\BB_{\Set}}$. For $k \in [n]$ define the events
	\[
		\BB_{k} = \left\{ \exists \Set \subseteq V(\Gnp) \quad |\Set|=k \, \wedge \, \eSet > \set\left(\set + C \dd^{-1/2}\right)  \frac{n\dd}{2}\right\}.
	\]
	Obviously
	\begin{equation} \label{eq:prob_sum}
		\Pra{\BB_{\Set}} = \Pra{\BB_{1} \cup \BB_{2} \cup \ldots \cup \BB_n} \le \sum_{k=1}^{n} \Pra{\BB_{k}}.
	\end{equation}
	
	The methods for upper bounding the parts of the above sum depend on the size of a set~$\Set$, thus on the particular range of $k$. Therefore the next part of the proof will be split into three cases corresponding to three different ranges of $k$. To define those ranges, first choose any function $\delta = \delta(n)$ satisfying $1/n \ll \delta \ll \dd^{-1}$ (note that such $\delta$ exists by the assumption $\dd\ll n$ from the statement of the theorem). For $s \in [0,1]$ let also  $X_s \sim \Bin{\binom{s n}{2}}{p}$.
	\begin{enumerate}
		\item {\bf Case $\mathbf{1 \le k \le \delta n}$.}\\
		In this case we consider sets $\Set \subseteq V(\Gnp)$ such that $|\Set| = \set n$ for $0< \set = \set(n) \le \delta$. Recall that $e(\Set) \sim \Bin{\binom{\set n}{2}}{p}$. 
		Therefore by the union bound we get
		\begin{align*}
			\Pra{\BB_{\set n}}
			&\le\binom{n}{\set n} \Pra{X_\set > \frac{\set^2\dd n}{2} +  \frac{C \set\dd^{1/2} n}{2}}\\
			&\le \left(\frac{en}{\set n}\right)^{\set n} \Pra{X_\set >  \frac{C \set\dd^{1/2} n}{2}}\\
			&\le \left(\frac{e}{\set }\right)^{\set n} \binom{\binom{\set n}{2}}{\frac{C}{2}\set \dd^{1/2}n}p^{\frac{C}{2}\set \dd^{1/2}n}\\
			&\le
			\left(\frac{e}{\set}\right)^{\set n}\left(\frac{e\set^2\dd n}{C\set \dd^{1/2}n}\right)^{\frac{C}{2}\set\dd^{1/2}n}\\
			&=
			\left(
			\frac
			{e^{1+\frac{C}{2}\dd^{1/2}}}
			{C^{\frac{C}{2}\dd^{1/2}}}
			\set^{\frac{C}{2}\dd^{1/2}-1}\dd^{\frac{C}{4}\dd^{1/2}}
			\right)^{\set n}\\
			&\le
			\left(
			e\left(\frac{e}{C}\right)^{\frac{C}{2}\dd^{1/2}}
			\delta^{\frac{C\dd^{1/2}}{4}-1}
			\right)^{\set n}\\
			&=
			\left(
			\frac{e^3}{C^2}\left(\frac{e^2}{C^2}\delta\right)^{\frac{C}{4}\dd^{1/2}-1}
			\right)^{\set n},
		\end{align*}
		where the last inequality follows by the fact that $\set \le \delta$, $\delta \le \dd^{-1}$ (i.e., $\dd\le \delta^{-1}$), and $C \ge 1.999$ and $\dd \ge 16.17$ which implies $\frac{C}{2}\dd^{1/2}-1>0$. Now set $\qq = \qq(n) = \frac{e^3}{C^2}\left(\frac{e^2}{C^2}\delta\right)^{\frac{C}{4}\dd^{1/2}-1}$. Note that $C \ge 1.999$ and $\dd \ge 16.17$ implies $\frac{C}{4}\dd^{1/2}-1>0$, and recall that $\delta = o(1)$, thus we have $\qq=o(1)$ and
		\begin{equation} \label{eq:sum_part1}
			\sum_{k=1}^{\delta n}\Pra{\BB_k}\le\sum_{k=1}^{\delta n}\qq^{k}=o(1).
		\end{equation}

		\item {\bf Case $\mathbf{\delta n\le k \le n/3}$.} \\
		In this case we consider sets $\Set \subseteq V(\Gnp)$ such that $|\Set| = \set n$ for $\delta \le \set = \set(n) \le 1/3$. We have
		
		\begin{align} \label{Eq:bin_standard}
			\binom{n}{\set n}
			&\le \left(\frac{e n}{\set n}\right)^{\set n} = \exp\left(\set n \left(\ln{\frac{1}{\set}}+1\right)\right).
		\end{align}

		Note that $\E X_\set = \frac{\set^2\dd n}{2}-\frac{\set\dd}{2}=(1+o(1))\frac{\set^2\dd n}{2}$. 
		Let the function $\varphi:[0,\infty) \rightarrow \mathbb{R}$ be defined as in Lemma~\ref{Lemma:Chernoff} (thus it is increasing). By \eqref{Eq:bin_standard}, Chernoff's inequality (see \eqref{Eq:Chernoff_upper_tail} in Lemma~\ref{Lemma:Chernoff}) and the union bound we get 
		\begin{align}
			\Pra{\BB_{\set n}}&\le\binom{n}{\set n} \Pra{X_\set > \frac{\set^2\dd n}{2} +  \frac{C \set\dd^{1/2} n}{2}
			} \nonumber\\
			&\le\binom{n}{\set n} \Pra{X_\set > \E X_\set +  \frac{C \set\dd^{1/2} n}{2}} \nonumber\\
			&\le \exp\left(\set n \left(\ln{\frac{1}{\set}}+1\right)\right) \exp\left(-(1+o(1)) \frac{\set^2\dd n}{2} \varphi\left( \frac{C}{s\dd^{1/2}} \right)  \right) \nonumber\\
			&\le \exp\left( - (1+o(1)) \set n \left(  \frac{\set\dd }{2} \varphi\left( \frac{C}{s\dd^{1/2}} \right) - \left(\ln{\frac{1}{\set}}+1\right) \right) \right)\\
			& = \exp\left( - (1+o(1)) \set n f(\set \dd^{1/2},\dd^{1/2},C) \right)
		\end{align}
		for the function $f$ defined as in Lemma \ref{Lemma_f(x,d)}. By Lemma \ref{Lemma_f(x,d)}, for $\set \leq 1/3$, $C \geq 1.999$, $\dd \ge 16.17$ (thus $\dd^{1/2} \ge 3.95$)
		we have that $f(\set \dd^{1/2},\dd^{1/2},C) > 0.001$. Since $\delta n\to \infty$ as $n\to\infty$, we obtain
		(recall that $\set = k/n$)
		\begin{equation} \label{eq:sum_part3}
			\sum_{k=\delta n}^{n/3}\Pra{\BB_k} \le\sum_{k=\delta n}^{n/3}\left(e^{- 0.001}\right)^k =o(1).
		\end{equation}

		\item {\bf Case $\mathbf{n/3\le k \le n}$.}\\
		In this case we consider sets $\Set \subseteq V(\Gnp)$ such that $|\Set| = \set n$ for $1/3 \le \set = \set(n) \le 1$. 
		Moreover $\E X_\set = \frac{\set^2\dd n}{2}-\frac{\set\dd}{2}=(1+o(1))\frac{\set^2\dd n}{2}$. Therefore by Chernoff's inequality (see \eqref{Eq:Chernoff_upper_tail} in Lemma~\ref{Lemma:Chernoff}) for any $\Set$ from this case we may write (recall that $e(\Set) \sim \Bin{\binom{\set n}{2}}{p}$) 
		
		\begin{align*}
			\Pra{e(\Set) > \frac{\set^2\dd n}{2} +  \frac{C \set\dd^{1/2} n}{2}} &= 
			\Pra{X_\set > \frac{\set^2\dd n}{2} +  \frac{C \set\dd^{1/2} n}{2}}\\
			&\le \Pra{X_\set > \E X_\set + \frac{C\set\dd^{1/2} n}{2}}\\
			&\le \exp\left(-(1+o(1)\frac{\set^2\dd n}{2}\varphi\left(\frac{C}{\set\dd^{1/2}}\right)\right)\\
			&= \exp\left(-(1+o(1))g(\set\dd^{1/2},C)n\right)\\
			&=o(2^{-n}), 
		\end{align*}
		where $g$ is defined as in Lemma~\ref{Lemma_f(x,d)}. By Lemma~\ref{Lemma_f(x,d)}, for $C \ge 1.999$, $\set \geq 1/3$, and $\dd \ge 16.17$ (i.e., $\set\dd^{1/2} \ge 1.34$) we get $g(\set\dd^{1/2},C) > \ln{2} + 0.01$, which implies the last equality. There are at most $2^n$ subsets of $V(\Gnp)$ to consider in this case (and in general) thus by the union bound we get
		\begin{equation} \label{eq:sum_part4}
			\sum_{k=n/3}^{n}\Pra{\BB_{k}}=o(1).
		\end{equation}

	\end{enumerate}

	By \eqref{eq:prob_sum}, \eqref{eq:sum_part1},  \eqref{eq:sum_part3}, and \eqref{eq:sum_part4} we get $\Pra{\BB_{\Set}} = o(1)$. Since in the above analysis we have considered the whole range of values for $|\Set|$, it immediately implies also $\Pra{\BB_{\Setbar}} = o(1)$.
	
	Now, only $\Pra{\BB_{\Set\Setbar}}$ is left to consider. For $s \in [0,1]$ let $Y_s \sim \Bin{s(1-s)n^2}{p}$, thus $\E[Y_s] = s(1-s)n^2 p$. Note that for any  $\Set \subseteq V(\Gnp)$ such that $|\Set| = \set n$ we have $e(\Set,\Setbar) \sim \Bin{\set(1-\set)n^2}{p}$, therefore by Chernoff's inequality (see \eqref{Eq:Chernoff_lower_tail} in Lemma \ref{Lemma:Chernoff}) 
	\begin{align*}
		& \Pra{\eSetSetbar  < \set(1-\set)\dd n-C\sqrt{\set(1-\set)}\dd^{1/2}n} \\
		& = \Pra{Y_\set < \E Y_{\set} -C\sqrt{\set(1-\set)}\dd^{1/2}n}
		\\
		&\le \exp\left(
		-\frac{C^2\set(1-\set)\dd n^2}{2\set(1-\set)\dd n}
		\right) =o(2^{-n}),
	\end{align*}
	where the last inequality follows by the fact that $C>\sqrt{2\ln 2}\approx 1.17741$. Taking the union bound over all $2^n$ subsets of $V(\Gnp)$, we get
	\[
		\Pra{\BB_{\Set\Setbar}}  = o(1).
	\]
\end{proof}

\begin{remark} \label{Rem:d_to_infty}
	As already mentioned, the bounds chosen for $C$ and $\dd$ in Lemma~\ref{Lem:Fluktuacje} ($C \ge 1.999$, $\dd \ge 16.17$)  are dictated by some trade-off between the final value in the upper bound on the modularity of $\Gnp$ and the smallest value of $\dd$ for which the proof is valid. One can verify that the proof of Lemma~\ref{Lem:Fluktuacje} is also valid by the assumption $C>2 \sqrt{\ln{2}} \approx 1.665$ when $\dd \to \infty$ as $n \to \infty$. This particular bound for $C$ then comes from the third case in the proof ($n/3 \leq k \leq n$) where we require that the increasing function $g$ satisfies $g(\set\dd^{1/2},C) > \ln{2}$. But at the same time, given fixed $C$ and $\set$, $g(\set\dd^{1/2},C) \to C^2/4$ as $\dd \to \infty$ thus $C$ has to exceed $2 \sqrt{\ln{2}}$.
\end{remark}

\section{Lower bound on the modularity of $\mathbf{G(n,p)}$ -- proof of Theorem \ref{Thm:Main2}} \label{Sec:lower_bound}

To get the lower bound on $\modu(\Gnp)$ we take advantage of the estimate of the minimum bisection of $\Gnp$ from \cite{Dembo2017}. There one finds also an analogous estimate for a random regular graph. The latter was used by McDiarmid and Skerman to get the lower bound on modularity of random regular graphs in \cite{McDiarmidScerman2018}.


\begin{theorem}[Theorem 1.2 of \cite{Dembo2017}]\label{Tw:MinCutGnm}
	Let $n \in \mathbb{N}$, $c\in {\mathbb R}_+$ be a constant, and let $G$ be a random graph chosen uniformly at random from all graphs  with the vertex set $[n]$ and with $\lfloor cn\rfloor$ edges. Then with high probability, for large $c$  
	\[
	\min_{\substack{\Set \subseteq V(G)\\ |\Set|-|\bar{\Set}|\in \{0,1\}}}
	\eSetSetbar=
	\left(\frac{c}{2}- P_*\sqrt{\frac{c}{2}}+o_{c}(\sqrt{c})
	\right)n,
	\]
	where $P_*=0.76321 \pm 0.00003 $ and $o_c(\cdot)$ relates to $c\to\infty$.  
\end{theorem}

It is mentioned in \cite{Dembo2017} that the above result implies an analogue for a random graph $\Gnp$. For completeness, we formulate it as a corollary and briefly sketch the justification.

\begin{corollary}\label{Cor:MinCutGnp}
	Let $n \in \mathbb{N}$, $\dd\in {\mathbb R}_+$, $p = p(n) = \dd/n$, and $G(n,p)$ be a binomial random graph. 
	Then with high probability, for large $\dd$
	\begin{equation}\label{Eq:MinCutGnp}
		\min_{\substack{\Set \subseteq V(\Gnp)\\ |\Set|-|\bar{\Set}|\in \{0,1\}}}
		e(\Set,\bar{\Set})=
		\left(\frac{\dd}{4}-P_*\frac{\sqrt{\dd}}{2}+o_{\dd}(\sqrt{\dd})
		\right)n,
	\end{equation}
	where $P_*=0.76321 \pm 0.00003 $ and $o_{\dd}(\cdot)$ relates to $\dd\to\infty$. 
\end{corollary}

\begin{proof}[Sketch of the proof]
	First, recall that $e(\Gnp) \sim \Bin{\binom{n}{2}}{p}$, i.e. $\E[ e(\Gnp)] = \frac{n\dd }{2}-\frac{\dd}{2}=(1+o(1))\frac{n\dd}{2}$. By the concentration of the binomial random variable we get that with high probability
	\begin{equation}\label{Eq:GnpEdgeConc}
		|e(\Gnp)-n\dd/2|\le \omega \sqrt{n\dd}.
	\end{equation}
	Now, let $c \in {\mathbb R}_+$ be a constant and $\omega$ a function tending arbitrarily slowly to infinity. Under the condition $e(\Gnp) = \lfloor cn \rfloor$, the graph $\Gnp$ is uniformly distributed over all graphs with vertex set $[n]$ and exactly $\lfloor cn \rfloor$ edges. 
	Apply Theorem~\ref{Tw:MinCutGnm} for all values of $\lfloor cn \rfloor$ such that $c_- \le c \le c_+$, where $c_{\pm} = \dd/2 \pm \omega \sqrt{\dd/ n}$. Finally, use a union bound over all possible values for the number of edges in $\Gnp$ and apply \eqref{Eq:GnpEdgeConc} to get the conclusion. For a more detailed analysis, we refer the reader to Section~1.4 in~\cite{JLRBook}.
\end{proof}

To prove Theorem \ref{Thm:Main2} we also need a concentration result for the number of edges within the vertex sets of $\Gnp$ whose cardinality is about $n/2$. It is presented in Lemma \ref{Lem:Fluktuacje2} below. 

\begin{lemma} \label{Lem:Fluktuacje2}
	Let $n \in \mathbb{N}$, $\dd\ge 1$, $p = p(n) = \dd/n$, and $G(n,p)$ be a binomial random graph. Let also $C> (7 \ln{2})/6 \approx 0.81$. Then {\whp} for all $\Set\subseteq V(\Gnp)$ such that $|\Set| \in \{\lfloor n/2 \rfloor, \lceil n/2 \rceil\}$
	\[
		\left|e(S)-\frac{n \dd}{8}\right| \leq C n \sqrt{\dd}.
	\] 
\end{lemma}
\begin{proof}
	Let $Z \sim \Bin{\binom{\lfloor n/2\rfloor}{2}}{p}$, i.e., $\E Z = \frac{n \dd}{8} + o(n \sqrt{\dd})$. For a fixed $\Set\subseteq V(\Gnp)$ such that $|\Set| = \lfloor n/2 \rfloor$, by Chernoff's inequality (see Lemma \ref{Lemma:Chernoff}) we have
	\begin{align*}
		\Pra{\left|e(S)-\frac{n \dd}{8}\right| \leq C n \sqrt{\dd}} & \le \Pra{|Z-\E Z| \le (1+o(1))Cn\sqrt{\dd}}\\
		& \le 2 \exp\left(-\frac{(1+o(1))C^2n^2\dd}{2\left( n\dd/8 + Cn\sqrt{\dd}/3\right)}\right)\\
		& =  2 \exp\left(-(1+o(1)) \frac{C^2}{\frac{1}{4} + \frac{2C}{3 \sqrt{\dd}}}n\right) \\
		& \le 2 e^{- (\ln{2}+0.1) n} = o(2^{-n}),
	\end{align*}
	where the last inequality is true by $C > (7 \ln{2})/6$ and $\dd\ge 1$. Since there are at most $2^n$ sets $\Set\subseteq V(\Gnp)$ such that $|\Set| \in \{\lfloor n/2 \rfloor, \lceil n/2 \rceil\}$, the conclusion follows by the union bound.
\end{proof}

\begin{proof}[Proof of Theorem \ref{Thm:Main2}]
	Let $\Part =\{\Set,\Setbar\}$ be a partition of $V(\Gnp)$ minimizing $\eSetSetbar$ over $\Set$ such that $|\Set|-|\bar{\Set}|\in \{0,1\}$.
	Then
	\begin{equation}\label{Eq:ErrDef}
		\begin{split}
			\eSet&=\frac{n\dd}{8}+\err_1+\err_0,\\
			\eSetbar&=\frac{n\dd}{8}+\err_2,\\
			\eSetSetbar&=\frac{n\dd}{4}-(\err_1+\err_2),\\
		\end{split}
	\end{equation}
	where $\err_0 = e(\Gnp) - n\dd/2$, and $\err_1$ and $\err_2$ are random variables representing the fluctuations in the number of edges within the sets $\Set$ and $\Setbar$.	Therefore, by Corollary \ref{Cor:MinCutGnp}, \eqref{Eq:GnpEdgeConc}, and Lemma~\ref{Lem:Fluktuacje2}, for large $\dd$, with high probability, we get 
	\begin{equation}\label{Eq:ErrWzory}
		\begin{split}
			\err_0&=o(n\sqrt{\dd}) ,	\\
			\err_1+\err_2&\ge \left( P_*\frac{\sqrt{\dd}}{2}+o_{\dd}(\sqrt{\dd})\right)n,\\
			|\err_i|&\le Cn\sqrt{\dd}, \text{ for }i=1,2\text{ and  } C\ge 0.81.
		\end{split}
	\end{equation}
	Now, by \eqref{Eq:ErrDef} and \eqref{Eq:ErrWzory} it follows that  for large $\dd$, with high probability
	\[
	\begin{split}
		4\eSet\eSetbar-\eSetSetbar^2
		&=
		\frac{n\dd}{2}\left(\err_1+\err_2\right)+4\err_1\err_2
		+\frac{n\dd}{2}\left(\err_1+\err_2\right)-(\err_1+\err_2)^2\\
		&\quad +\err_0\left(\frac{n\dd}{2}+4\err_2\right)\\
		&=n\dd(\err_1+\err_2)-(\err_1-\err_2)^2+o(n^2\dd^{3/2})\\
		&\ge \left(\frac{{P_*}+o_{\dd}(1)}{2\sqrt{\dd}}-\frac{4C^2}{\dd}+o\left(\frac{1}{\sqrt{\dd}}\right)\right)n^2\dd^2\\
		&=\frac{{P_*}+o_{\dd}(1)}{2\sqrt{\dd}}n^2\dd^2
	\end{split}
	\]
	and, by \eqref{Eq:GnpEdgeConc} and Fact \ref{Fact:mod_as_edges} we finally obtain
	\[
	\modu(\Gnp) \ge  \modu_{\Part}(\Gnp) \ge 2\frac{{P_*}+o_{\dd}(1)}{2\sqrt{\dd}}(1+o(1))=\frac{{P_*}+o_{\dd}(1)}{\sqrt{\dd}}.
	\]
\end{proof}



\bibliographystyle{plain}
\bibliography{modularity_Gnp}

\begin{thebibliography}{10}

\bibitem{Agdur2023}
V.\ Agdur, N.\ Kam{\v{c}}ev, and F.\ Skerman.
\newblock Universal lower bound for community structure of sparse graphs, 2023.

\bibitem{BlGuLaLe08}
V.D. Blondel, J.L. Guillaume, R.~Lambiotte, and E.~Lefebvre.
\newblock Fast unfolding of communities in large networks.
\newblock {\em Journal of Statistical Mechanics: Theory and Experiment},
  2008(10):P10008, 2008.

\bibitem{Bolla2015}
M.~Bolla, B.~Bullins, S.~Chaturapruek, S.~Chen, and K.~Friedl.
\newblock Spectral properties of modularity matrices.
\newblock {\em Linear Algebra Appl.}, 473:359–--376, 2015.

\bibitem{Brandes2008}
U.~Brandes, D.~Delling, M.~Gaertler, R.~Gorke, M.~Hoefer, Z.~Nikoloski, and
  D.~Wagner.
\newblock On modularity clustering.
\newblock {\em IEEE Transactions on Knowledge and Data Engineering},
  20(2):172--188, 2008.

\bibitem{ChFoSk21}
J.~Chellig, N.~Fountoulakis, and F.~Skerman.
\newblock The modularity of random graphs on the hyperbolic plane.
\newblock {\em Journal of Complex Networks}, 10(1):cnab051, 2021.

\bibitem{ChungBook}
F.~Chung.
\newblock {\em Spectral Graph Theory}, volume~92.
\newblock American Mathematical Society, 1997.

\bibitem{Chung2003}
F.~Chung, L.~Lu, and V.~Vu.
\newblock The spectra of random graphs with given expected degrees.
\newblock {\em Internet mathematics}, 1(3):257--275, 2003.

\bibitem{Coja-Oghlan2007}
A.~Coja-Oghlan.
\newblock On the {L}aplacian eigenvalues of gn,p.
\newblock {\em Combinatorics, Probability and Computing}, 16(6):923–946,
  2007.

\bibitem{Dembo2017}
A.~Dembo, A.~Montanari, and S.~Sen.
\newblock Extremal cuts of sparse random graphs.
\newblock {\em Annals of Probability}, 45(2):1190--1217, 2017.

\bibitem{GiSh_23}
G.~Gilad and R.~Sharan.
\newblock From {L}eiden to {T}el-{A}viv {U}niversity {(TAU)}: exploring
  clustering solutions via a genetic algorithm.
\newblock {\em PNAS Nexus}, 2(6):pgad180, 2023.

\bibitem{JLRBook}
S.~Janson, T.~{\L}uczak, and A.~Ruci\'{n}ski.
\newblock {\em Random Graphs}.
\newblock Wiley, 2001.

\bibitem{Kaminski2019}
B.~Kami{\'n}ski, V.~Poulin, P.~Pra{\l}at, P.~Szufel, and F.~Th{\'e}berge.
\newblock Clustering via hypergraph modularity.
\newblock {\em Plos One}, 14:e0224307, Feb 2019.

\bibitem{KaminskiBook}
B.~Kami{\'n}ski, P.~Pra{\l}at, and F.Th{\'e}berge.
\newblock {\em Mining Complex Networks}.
\newblock Chapman and Hall/CRC, 2021.

\bibitem{LaSu23}
M.~Laso{\'n} and M.~Sulkowska.
\newblock Modularity of minor-free graphs.
\newblock {\em Journal of Graph Theory}, 102(4):728--736, 2023.

\bibitem{LiMi22}
L.~Lichev and D.~Mitsche.
\newblock On the modularity of 3-regular random graphs and random graphs with
  given degree sequences.
\newblock {\em Random Structures \& Algorithms}, 61(4):754--802, 2022.

\bibitem{Love1980Log}
E.R. Love.
\newblock 64.4 {S}ome logarithm inequalities.
\newblock {\em The Mathematical Gazette}, 64(427):55--57, March 1980.

\bibitem{Majstorovic2014}
S.~Majstorovi{\'c} and D.~Stevanovi{\'c}.
\newblock A note on graphs whose largest eigenvalues of the modularity matrix
  equals zero.
\newblock {\em Electron. J. Linear Algebra}, 27:256, 2014.

\bibitem{McDRSS_2024}
C.~McDiarmid, K.~Rybarczyk, F.~Skerman, and M.~Sulkowska.
\newblock Expansion and modularity in preferential attachment graph, 2025.
\newblock Preprint.

\bibitem{McDiarmidScerman2018}
C.~McDiarmid and F.~Skerman.
\newblock Modularity of regular and treelike graphs.
\newblock {\em Journal of Complex Networks}, 6(4):596--619, 2018.

\bibitem{McDiarmid2020}
C.~McDiarmid and F.~Skerman.
\newblock Modularity of {E}rdős‐{R}ényi random graphs.
\newblock {\em Random Structures and Algorithms}, 57(1):211--243, 2020.

\bibitem{McDiarmidSkerman_dense2023}
C.~McDiarmid and F.~Skerman.
\newblock Modularity of nearly complete graphs and bipartite graphs, 2023.

\bibitem{NewmanBook}
M.E.J. Newman.
\newblock {\em Networks}.
\newblock Oxford University Press, 2 edition, 2018.

\bibitem{Newman2004}
M.E.J. Newman and M.~Girvan.
\newblock Finding and evaluating community structure in networks.
\newblock {\em Physical Review E}, 69(2):026113, February 2004.

\bibitem{Prokhorenkova2017}
L.\ Prokhorenkova, A.\ Raigorodskii, and P.\ Pra{\l}at.
\newblock Modularity of complex networks models.
\newblock {\em Internet Mathematics}, 2017.

\bibitem{Reichardt2006}
J.~Reichardt and S.~Bornholdt.
\newblock When are networks truly modular?
\newblock {\em Physica D: Nonlinear Phenomena}, 224(1):20--26, 2006.
\newblock Dynamics on Complex Networks and Applications.

\bibitem{Rybarczyk2025randomintersectiongraphs}
K.~Rybarczyk.
\newblock Modularity of random intersection graphs, 2025.
\newblock Accepted for 20th Workshop on Modelling and Mining Networks, WAW
  2025, Vilnius, Lithuania, June 30 – July 3, 2025.

\bibitem{Rybarczyk2025}
K.~Rybarczyk and M.~Sulkowska.
\newblock Modularity of preferential attachment graphs, 2025.

\bibitem{TrWaEc19}
V.A. Traag, L.~Waltman, and N.J. van Eck.
\newblock From {L}ouvain to {L}eiden: guaranteeing well-connected communities.
\newblock {\em Scientific Reports}, 9(5233), 2019.

\end{thebibliography}

\appendix

\section{Proof of Lemma~\ref{Lemma_f(x,d)}}

We formulate two facts before we move on to the proof of Lemma~\ref{Lemma_f(x,d)}. The first one presents standard inequalities for the logarithm function (see \cite{Love1980Log}).

\begin{fact}
We have
\begin{align}
	\label{Eq:LogLower}
	\ln(1+t)&>\frac{t}{1+t/2}&&\text{ for }t>0,\\
	\label{Eq:LogUpper} 
	\ln(1+t)&< t&&\text{ for }t>-1\text{ and }t\neq 0.
\end{align}
\end{fact}

The second one analyzes monotonicity of several functions.
\begin{fact}\label{Fact:Increasing}
	Let $\varphi(t) = (1+t)\ln{(1+t)}-t$ for $t>0$ (thus $\varphi$ is defined as in Lemma~\ref{Lem:Fluktuacje}). Then
	\begin{enumerate}
		\item $\varphi(t)$ is increasing in $t>0$,
		\item $g(x,z)=\frac{x^2}{2}\varphi\left(\frac{z}{x}\right)$ is increasing in $x>0$ for any $z>0$, and increasing in $z>0$ for any $x>0$,
		\item $h_1(x) = x(\ln{\left(1+\frac{z}{x}\right)}-\frac{z}{x})$ is increasing in $x>0$ for any $z>0$,
		\item $h_2(y) = y^2 \left(\ln{\left(1+\frac{3z}{y}\right)-\frac{3z}{y}}\right)$ is decreasing in $y>0$ for any $z>0$,
		\item $h_3(t)=\ln(1+t)-t$ is decreasing in $t>0$.
	\end{enumerate}
\end{fact}
\begin{proof}

\begin{enumerate}
	\item Obviously $\varphi(t)$ is increasing in $t>0$ as $\varphi'(t)=\ln(1+t)>0$ for $t>0$.  

	\item We consider the function
\[
g(x,z)=\frac{x^2}{2}\varphi\left(\frac{z}{x}\right)=\frac{x^2}{2}\left(\left(1+\frac{z}{x}\right)\ln\left(1+\frac{z}{x}\right)-\frac{z}{x}\right),\quad x,z>0.
\]
Given $x>0$, it is increasing in $z>0$, as $\varphi(z/x)$ is increasing in $z$ and $x^2>0$ is set. 

Now set $z>0$ and define ${\bar g}(x)=  g(x,z)$. Using \eqref{Eq:LogLower}, we get
\[
2{\bar g}'(x)=(2x+z)\ln\left(1+\frac{z}{x}\right)-2z>(2x+z)\frac{\frac{z}{x}}{1+\frac{z}{2x}}-2z=0,
\]
thus $g(x,z)$ is increasing in $x>0$ for any $z>0$.
	
	\item Set $z>0$. Note that by \eqref{Eq:LogUpper} for $x>0$
	\[
		h_1'(x)=-\ln\left(1-\frac{z}{z+x}\right)-\frac{z}{z+x}>0
	\]
	thus $h_1(x)$ is increasing in $x>0$ for any $z>0$.
	
	\item Set $z>0$. We have
	\[
		h_2'(y) = 2y \ln{\left( 1 + \frac{3 z}{y} \right)} - \frac{3z(2y+3z)}{y+3z}
	\]
	and, by \eqref{Eq:LogLower},
	\[
		h_2''(y) =  2 \ln{\left( 1 + \frac{3 z}{y} \right)} - \frac{3z(2y+9z)}{(y+3z)^2} > \frac{2 \frac{3z}{y}}{1+\frac{3z}{2y}} - \frac{3z(2y+9z)}{(y+3z)^2} > 0 \quad \text{ for } \quad y>0.
	\]
	Thus $h_2'(y)$ is increasing in $y>0$ and, since $\lim_{y \to 0} h_2'(y) = -3z<0$ and $\lim_{y \to \infty} h_2'(y) = 0$, we get that $h_2'(y)<0$ for $y>0$. Consequently, $h_2(y)$ is decreasing in $y>0$ for any $z>0$.
	
	\item Obviously, $h_3(t)$ is decreasing in $t>0$ as $h_3'(t)=\frac{1}{1+t}-1<0$ for $t>0$.
\end{enumerate}
\end{proof}

\begin{proof}[Proof of Lemma \ref{Lemma_f(x,d)}]

Recall that for $x,y,z>0$
\[
f(x,y,z)
=\frac{xy}{2}\varphi\left(\frac{z}{x}\right)-\ln\frac{y}{x}-1
=\frac{xy}{2}\left(\left(1+\frac{z}{x}\right)\ln\left(1+\frac{z}{x}\right)-\frac{z}{x}\right)-\ln\frac{y}{x}-1.
\]
Let also $h_1$ and $h_2$ be defined as in Fact \ref{Fact:Increasing}. Set $y,z>0$ and let $\bar{f}:(0,y/3]\to \mathbb{R}$ be given by
\[
\bar{f}(x):=f(x,y,z).
\]
Then
\[
\bar{f}'(x)=\frac{y x\left(\ln\left(1+\frac{z}{x}\right)-\frac{z}{x}\right)+2}{2x} = \frac{yh_1(x)+2}{2x}.
\]
By Fact \ref{Fact:Increasing}(iii) $h_1(x)$ is increasing in $x$ for any $z>0$. Therefore for $0<x\le y/3$, when $z = 1.999$ and $y\ge3.95$ we have
\begin{equation} \label{eq:1}
y h_1(x)+2
\le y h_1\left(\frac{y}{3}\right)+2
= \frac{y^2}{3}\left(\ln\left(1+\frac{3z}{y}\right)-\frac{3z}{y}\right)+2 = \frac{1}{3}h_2(y)+2<0,
\end{equation}
where the last inequality follows by Fact~\ref{Fact:Increasing}(iv) and the fact that $\frac{1}{3}h_2(3.95) + 2 < 0$ by $z=1.999$.
Hence $\bar{f}(x)$ is decreasing in $x>0$ and, when $z = 1.999$ and $y \ge 3.95$, for $1\le x\le y/3$, by Fact~\ref{Fact:Increasing}(ii), we get
\begin{equation} \label{eq:2}
\begin{split}
\bar{f}(x)\ge  \bar{f}\left(\frac{y}{3}\right) &= \frac{y^2}{6}\left(\left(1+\frac{3z}{y}\right)\ln\left(1+\frac{3z}{y}\right)-\frac{3z}{y}\right)-\ln3-1\\
& =3\, g\left(y/3,z\right)-\ln3-1\ge 3\, g\left(3.95/3,1.999\right) -\ln3-1> 0.001.
\end{split}
\end{equation}
Note that if \eqref{eq:1} and \eqref{eq:2} are true for $z=1.999$, by Fact~\ref{Fact:Increasing}(i) and Fact~\ref{Fact:Increasing}(v) they are also true for any $z \ge 1.999$. This proves the first part of the lemma.

Now let us consider $g(x,z)=\frac{x^2}{2}\varphi\left( \frac{z}{x} \right)$ for  $x,z>0$.
By Fact~\ref{Fact:Increasing}(ii), for $x \ge 1.34$ and $z \ge 1.999$,  we simply get
\[
g(x,z)\ge h(1.34,1.999)  > \ln{2} + 0.01.
\]
This proves the second part of the lemma.


\end{proof}

\end{document}